\documentclass{endm}
\usepackage{endmmacro}
\usepackage{graphicx}
\usepackage{amsmath}

\def\Str{\mathop{\mathrm{Str}}\nolimits}
\def\dom{\mathop{\mathrm{Dom}}\nolimits}

\def\str#1{\mathbf {#1}}

\def\arity#1{a(\rel{}{#1})}
\def\nbrel#1#2{R_{#1}^{#2}}
\def\rel#1#2{R_{\mathbf{#1}}^{#2}}
\def\func#1#2{F_{\mathbf{#1}}^{#2}}
\def\nbfunc#1#2{F_{#1}^{#2}}

\def\K{{\mathcal K}}
\def\Fraisse{Fra\"{\i}ss\' e}

\begin{document}

\begin{verbatim}\end{verbatim}\vspace{2.5cm}
\pagestyle{plain}
\begin{frontmatter}

\title{Ramsey theorem for designs}

\author{Jan Hubi\v cka\thanksref{ERC}\thanksref{hubicka@iuuk.mff.cuni.cz}}
\author{Jaroslav Ne\v set\v ril\thanksref{ERC}\thanksref{nesetril@iuuk.mff.cuni.cz}}
\address{Computer Science Institute of Charles University (IUUK)\\ Charles University\\ Prague, Czech Republic}

\thanks[ERC]{Supported by grant ERC-CZ LL-1201 of the Czech Ministry of Education and CE-ITI P202/12/G061 of GA\v CR.}
\thanks[hubicka@iuuk.mff.cuni.cz]{Email: \href{mailto:hubicka@iuuk.mff.cuni.cz} {\texttt{\normalshape{hubicka@iuuk.mff.cuni.cz}}}}
\thanks[nesetril@iuuk.mff.cuni.cz]{Email: \href{mailto:nesetril@iuuk.mff.cuni.cz} {\texttt{\normalshape{nesetril@iuuk.mff.cuni.cz}}}}
\begin{abstract}
We prove that for any choice of parameters $k,t,\lambda$ the class of all finite ordered designs with parameters $k,t,\lambda$ is a Ramsey class.
\end{abstract}
\end{frontmatter}
\section{Introduction}
We prove that for every choice of parameters $2\leq t\leq k$ and $1\leq \lambda$ the class $\overrightarrow{P\mathcal D}_{kt\lambda}$ of linearly
ordered partial designs with parameters $k,t,\lambda$ is a Ramsey class. Thus, together with the 
recent spectacular results of Keevash~\cite{Keevash2014}, one obtains that the class of linearly ordered designs $\overrightarrow{\mathcal D}_{kt\lambda}$ is a Ramsey class.

This paper involves three seemingly unrelated subjects: block designs, model theory and structural Ramsey theory.
The generality is an important issue and in such context our main result can be formulated as follows:
\begin{thm}
\label{thm:designs}
For any choice of parameters $k,t,\lambda$ the class $\overrightarrow{\mathcal D}_{kt\lambda}$ of all finite ordered designs with parameters $k,t,\lambda$ is a Ramsey class.
\end{thm}
 For the proof we have to find the right degree of abstraction which will be introduced in the next three sub-sections together with all relevant notions.
Strong structural Ramsey theorem (proved in \cite{Hubicka2016,Evans3}) plays the key role.

\subsection{Designs}
A {\em $(k,t,\lambda)$-design} ($k\geq t,\lambda$ all positive integers) is a finite hypergraph $(X,\rel{}{})$ where $\rel{}{}$ is a set of $k$-subsets of $X$ with property that any $t$-subset of $X$ is contained in exactly $\lambda$ elements of $\rel{}{}$. More formally we have $\rel{}{}\subseteq {X\choose k}$ and $\lvert \left\{M\in \rel{}{}:T\subseteq M\right\}\rvert=\lambda$ for any $T\in {X\choose t}$ (as usual in Ramsey context we denote by $X\choose k$ the set of all $k$-subsets of $X$ consisting of $k$ elements). A {\em partial $(k,t,\lambda)$-design} is hypergraph $(X,\rel{}{})$ where every $t$-subset is in at most $\lambda$ elements of $\rel{}{}$.

Designs form a classical area of combinatorics as well as of mathematical statistics (design of experiments). 
 Particularly Keevash~\cite{Keevash2014,Kalai2015}, extending another spectacular result in the area~\cite{wilsonI,wilsonII,wilsonIII}, recently showed the following:
\begin{thm}[Keevash theorem~\cite{Keevash2014}]
\label{thm:completion}
 For every choice of parameters $k,t,\allowbreak \lambda$ there exists $(k,t,\lambda)$-design on every sufficiently large set satisfying a well known divisibility condition. Also any partial $(k,t,\lambda)$-design can be completed to a $(k,t,\lambda)$-design.
\end{thm}
\subsection{Models}

Let $L=L_\mathcal R\cup L_\mathcal F$ be a language involving relational symbols $\rel{}{}\in L_\mathcal R$ and function symbols $F\in L_\mathcal F$ each having associated positive integers called {\em arity} and denoted by $\arity{}$ for relations and {\em domain arity}, $d(\func{}{})$, {\em range arity}, $r(\func{}{})$, for functions.  
An \emph{$L$-structure} $\str{A}$ is a structure with {\em vertex set} $A$, functions $\func{A}{}:\dom(\func{A}{})\to {A\choose {r(\func{}{})}}$, $\dom(\func{A}{})\subseteq A^{d(\func{}{})}$ for $\func{}{}\in L_\mathcal F$ and relations $\rel{A}{}\subseteq A^{\arity{}}$ for $\rel{}{}\in L_\mathcal R$.
$\dom(\func{A}{})$ is called the {\em domain} of function $\func{}{}$  in $\str{A}$.
Notice that the domain is set of ordered $d(\func{}{})$-tuples while the range is set of unordered $r(\func{}{})$-tuples.
Symmetry in ranges permits explicit description of algebraic closures in the \Fraisse{} limits without changing the automorphism group (c.f.~\cite{Evans2}). It also simplifies some of the notation bellow.

The language $L$ is usually fixed and understood from the context.  If set $A$ is finite we call \emph{$\str A$ finite structure} (in most of this paper all structures are finite). 
If language $L$ contains no function symbols, we call $L$ {\em relational language} and every $L$-structure is also called {\em relational $L$-structure}.

The notion of embeddings, isomorphism, homomorphisms and free amalgamation 
are natural generalisation of the corresponding notions on relational structures
and are formally introduced in Section~\ref{sec:background}. Considering function
symbols has important consequences to what we consider as a substructure:
An $L$-structure $\str{A}$  is a {\em substructure} of $\str{B}$ if $A\subseteq B$ and all relations and functions of $\str{B}$ restricted to $A$
are precisely relations and functions of $A$. In particular a $t$-tuple
$\vec{t}$ of vertices of $A$ is in $\dom(\func{A}{})$ if and only if it is also in
$\dom(\func{B}{})$ and $\func{A}{}(\vec{t})=\func{B}{}(\vec{t})$. 
This implies the fact that $\str{B}$ does not induce a substructure on every subset of $B$ (but only on ``closed'' sets, to be defined later).

In our setting $(k,t,\lambda)$-design corresponds to a particular $L$-structure with $L_\mathcal F=\emptyset$, $L_\mathcal R=\{\rel{}{}\}$, $a(\rel{}{})=k$. However designs satisfy some furhter properties and it is essential for our argument that we introduce and use function symbols.

\subsection{Ramsey classes}

For structures $\str{A},\str{B}$ denote by ${\str{B}\choose \str{A}}$ the set of all sub-structures of $\str{B}$, which are isomorphic to $\str{A}$.  Using this notation the definition of a Ramsey class gets the following form:
A class $\mathcal C$ is a \emph{Ramsey class} if for every two objects $\str{A}$ and $\str{B}$ in $\mathcal C$ and for every positive integer $k$ there exists a structure $\str{C}$ in $\mathcal C$ such that the following holds: For every partition ${\str{C}\choose \str{A}}$ into $k$ classes there exists an $\widetilde{\str B} \in {\str{C}\choose \str{B}}$ such that ${\widetilde{\str{B}}\choose \str{A}}$ belongs to one class of the partition.  It is usual to shorten the last part of the definition to $\str{C} \longrightarrow (\str{B})^{\str{A}}_k$.

We are motivated by the following, now classical, result.
\begin{thm}[Ne\v set\v ril-R\"odl theorem~\cite{Nevsetvril1976}]
\label{thm:NRoriginal}
Let $\str{A}$ and $\str{B}$ be a linearly ordered hypergraphs, then there exists a linearly ordered hypergraph
$\str{C}$ such that $\str{C}\longrightarrow (\str{B})^\str{A}_2$.

Moreover, if $\str{A}$ and $\str{B}$ do not contain an irreducible hypergraph
$\str{F}$  then $\str{C}$ may be chosen
with the same property (a hypergraph $F$ is irreducible if every pair of its elements is contained in an edge of $F$). 
\end{thm}
Given language $L$, denote by $\overrightarrow{L}$ language $L$ extended by one
binary relation $\leq$.  Given $L$-structure $\str{A}$ the {\em ordering of $\str{A}$}
is $\overrightarrow{L}$-structure extending $\str{A}$
by arbitrary linear ordering of vertices represented by $\leq_\str{A}$.
We denote such ordered $\str{A}$ as $\overrightarrow{\str{A}}$.
Given class $\K$ of $L$-structures denote by $\overrightarrow{\K}$ the class of
all orderings of structures in $\K$ i.e. $\overrightarrow{\K}$ is a class of $\overrightarrow{L}$-structures where $\leq$ is a linear order.
We sometimes say that $\overrightarrow{\K}$ arises by the {\em free orderings} of structures in $\mathcal K$.
For purposes of this paper Theorem~\ref{thm:NRoriginal} can now be re-formulated using notions of \Fraisse{} theory (which will be briefly introduced in Section~\ref{sec:background})
as follows:
\begin{thm}[Ramsey theorem for free amalgamation classes]
\label{thm:NR}
Let $L$ be a relational language,
$\K$ be a free amalgamation class of relational $L$-structures.  Then $\overrightarrow{\K}$ is a Ramsey class.
\end{thm}
Connection of Ramsey classes and extremely amenable
groups~\cite{Kechris2005} motivated a systematic search for new examples of Ramsey classes. It
became apparent that it is important to consider structures with both relations
and functions or, equivalently, classes of structures with ``strong embeddings''.
  This led to \cite{Hubicka2016} which provides a sufficient structural
condition for a subclass of a Ramsey class to be Ramsey and
generalises this approach also to classes with function symbols representing closures. 
Comparing the two main results of~\cite{Hubicka2016} (Theorem~2.1 for classes without closures and Theorem 2.2 for classes with closures) it is clear
that considering closures leads to many technical difficulties. In fact,
a recent example given in~\cite{Evans2} shows that there is no  direct analogy of Theorem~\ref{thm:NR}
to free amalgamation classes with closures.
Perhaps surprisingly, one can prove that if closures are explicitly represented by means
of partial functions, such statement is true. More precisely the following we proved in~\cite[Theorem 1.3]{Evans3} as a more streamlined version of Theorem 2.2 of~\cite{Hubicka2016}:
\begin{thm}
\label{thm:main}
Let $L$ be a language (involving relational symbols and partial functions),
$\K$ be a free amalgamation class of $L$-structures. 
 Then $\overrightarrow{\K}$ is a Ramsey class.
\end{thm}
It appears (see~\cite{Evans3}) that many natural classes may be interpreted as free amalgamation classes and consequently Theorem~\ref{thm:main} yields uniform proofs of Ramsey property of some recently discovered Ramsey classes (such as ordered partial Steiner systems (i.e. $(k,t,1)$-designs)~\cite{bhat2016ramsey}, bowtie-free graphs~\cite{Hubivcka2014}, bouquet-free graphs~\cite{Cherlin2007} and a Ramsey expansion
of class $2$-orientations of Hrushovski predimension construction~\cite{Evans2}). In this paper we add to this list of applications of Theorem~\ref{thm:main} yet another example from a very different area.

\section{Preliminaries}
\label{sec:background}
We now review some standard model-theoretic notions (see e.g.~\cite{Hodges1993}).

An \emph{embedding} $f:\str{A}\to \str{B}$ is an injective mapping $f:A\to B$ satisfying for every $\rel{}{}\in L_\mathcal R$ and $\func{}{}\in L_\mathcal F$:
\begin{enumerate}
\item $(x_1,x_2,\ldots, x_{\arity{}})\in \rel{A}{}\iff (f(x_1),f(x_2),\ldots,f(x_{\arity{}}))\in \rel{B}{}$, and,
\item $(x_1,\ldots, x_{d(\func{}{})})\in\dom(\func{A}{}) \iff (f(x_1),\ldots,f(x_{d(\func{}{})}))\in \dom(\func{B}{})$
 and\\ $f(\func{A}{}(x_1,c_2,\allowbreak \ldots, x_{d(\func{}{})}))=\func{B}{}(f(x_1),f(x_2),\ldots,f(x_{d(\func{}{})})).$
\end{enumerate}

  If $f$ is an embedding which is an inclusion then $\str{A}$ is a \emph{substructure} of $\str{B}$. For an embedding $f:\str{A}\to \str{B}$ we say that $\str{A}$ is \emph{isomorphic} to $f(\str{A})$ and $f(\str{A})$ is also called a \emph{copy} of $\str{A}$ in $\str{B}$. Thus $\str{B}\choose \str{A}$ is defined as the set of all copies of $\str{A}$ in $\str{B}$.
Given $\str{A}\in \K$ and $B\subset A$, the {\em closure of $B$ in $\str{A}$} is the smallest substructure of $\str{A}$ containing $B$.

Let $\str{A}$, $\str{B}_1$ and $\str{B}_2$ be structures with $\alpha_1$ an embedding of $\str{A}$
into $\str{B}_1$ and $\alpha_2$ an embedding of $\str{A}$ into $\str{B}_2$, then
every structure $\str{C}$
 together with embeddings $\beta_1:\str{B}_1 \to \str{C}$ and
$\beta_2:\str{B}_2\to\str{C}$ satisfying $\beta_1\circ\alpha_1 =
\beta_2\circ\alpha_2$ is called an \emph{amalgamation of $\str{B}_1$ and $\str{B}_2$ over $\str{A}$ with respect to $\alpha_1$ and $\alpha_2$}. 
We will call $\str{C}$ simply an \emph{amalgamation} of $\str{B}_1$ and $\str{B}_2$ over $\str{A}$
(as in the most cases $\alpha_1$, $\alpha_2$ and $\beta_1$, $\beta_2$ can be chosen to be inclusion embeddings).

Amalgamation is \emph{free} if $\beta_1(x_1)=\beta_2(x_2)$ if and only if $x_1\in \alpha_1(A)$ and $x_2\in \alpha_2(A)$
and there are no tuples in any relations  of $\str{C}$ and $\dom(\func{C}{})$, $\func{}{}\in L_\mathcal F$, using both vertices of
$\beta_1(B_1\setminus \alpha_1(A))$ and $\beta_2(B_2\setminus \alpha_2(A))$.
An \emph{amalgamation class} is a class $\K$ of finite structures satisfying the following three conditions:
\begin{enumerate}
\item {\em Hereditary property:} For every $\str{A}\in \K$ and a substructure $\str{B}$ of $\str{A}$ we have $\str{B}\in \K$;
\item {\em Joint embedding property:} For every $\str{A}, \str{B}\in \K$ there exists $\str{C}\in \K$ such that $\str{C}$ contains both $\str{A}$ and $\str{B}$ as substructures;
\item {\em Amalgamation property:} 
For $\str{A},\str{B}_1,\str{B}_2\in \K$ and $\alpha_1$ embedding of $\str{A}$ into $\str{B}_1$, $\alpha_2$ embedding of $\str{A}$ into $\str{B}_2$, there is $\str{C}\in \K$ which is an amalgamation of $\str{B}_1$ and $\str{B}_2$ over $\str{A}$ with respect to $\alpha_1$ and $\alpha_2$.
\end{enumerate}
If the $\str{C}$ in the amalgamation property can always be chosen as the free amalgamation, then $\K$ is {\em free amalgamation class}.

This explains all the notions in our Theorem~\ref{thm:main}. In the next section we apply this result to designs.

\section{Main result}
To deal with partial $(k,t,\lambda)$-designs in order to apply Theorem~\ref{thm:main} we need a particular encoding. Our language is denoted by $L=(L_\mathcal R,L_\mathcal F)$. The relational part $L_\mathcal R$ consists from relational symbol $\rel{}{}$ of arity $k$. We put $K=(k-t)\lambda+t$. The functional language $L_\mathcal F$ consists from symbol $\func{}{k},\func{}{k+1},\ldots,\func{}{K}$ all with domain arity $d(\func{}{\ell})=t$ and range arity $r(\func{}{\ell})=\ell$, $\ell=k,k+1,\ldots, K$.

 Denote by $\Str(L)$ the class of all $L$-structures (i.e. models of the language $L$). Within this class $\Str(L)$ we define a subclass $P\mathcal D_{kt\lambda}$ of all structures $\str{A}=(A,\rel{A}{},(\func{A}{\ell}: \ell=k,k+1,\ldots,K))$ which satisfy
\begin{enumerate}
  \item for every $(v_1,v_2,\ldots, v_k)\in \rel{A}{}$ it holds that $v_i\neq v_j$, $1\leq i<j\leq k$;
  \item if $(v_1,v_2,\ldots, v_k)\in \rel{A}{}$ then $(v_{\pi(1)},v_{\pi(2)},\ldots, v_{\pi(k)})\in \rel{A}{}$ for any permutation $\pi$;
  \item $(A,\widetilde{R}_\str{A})$ is partial $(k,t,\lambda)$-design where $\widetilde{R}_\str{A}=\{\{v_1,v_2,\ldots, v_k\}:(v_1,v_2,\allowbreak \ldots,\allowbreak v_k)\}\in \rel{A}{}$;
  \item $(v_1,v_2,\ldots, v_t)\in \dom(\func{A}{\ell})$ if and only if $\lvert N_\str{A}(v_1,v_2,\ldots, v_t)\rvert=\ell$ and $\func{A}{\ell}(v_1,v_2,\ldots,v_t)=N_\str{A}(v_1,v_2,\ldots, v_t)$ where $N_\str{A}(v_1,v_2,\ldots, v_t)$ is the {\em neighbourhood} of set $\{v_1,v_2,\ldots, v_t\}$ --- that is the set of all vertices $v$ such that there exists $M\in \widetilde{R}_\str{A}$ containing each of vertices $v,v_1,v_2,\ldots, v_t$.
\end{enumerate}
The embeddings in $P\mathcal D_{kt\lambda}$ are inherited from $\Str(L)$. Let us explicitly formulate their form:

For $\str{A}=(A,\rel{A}{},(\func{A}{\ell}: \ell=k,k+1,\ldots,K))$, $\str{B}=(B,\rel{B}{},(\func{B}{\ell}: \ell=k,k+1,\ldots,K))$ injective mapping $f:A\to B$ is an embedding of $\str{A}$ into $\str{B}$ if 
it satisfies the following conditions:
\begin{enumerate}
\item $(x_1,x_2,\ldots,x_k)\in \rel{A}{}$ if and only if $(f(x_1),f(x_2),\ldots, f(x_k))\in \rel{B}{}$ (i.e. $f$ is embedding $(A,\rel{A}{})$ into $(B,\rel{B}{})$),
\item for every $\ell$, $k\leq \ell\leq K$, it satisfies $\{f(x):x\in \func{A}{\ell}(x_1,x_2,\ldots, x_t)\}=\func{B}{\ell}(\{f(x_1),f(x_2),\allowbreak \ldots,\allowbreak  f(x_t)\})$ whenever one side of this equation make sense.
\end{enumerate}
 Every ordered partial $(k,t,\lambda)$-design $(X,\rel{}{},\leq)$ may be interpreted in $\overrightarrow{P\mathcal D}_{kt\lambda}$ (the class of free orderings of $P\mathcal D_{kt\lambda}$) as the following $L$-structure:
$\str{A}=(A,\rel{A}{},\leq_\str{A},(F^l_\str{A}:k\leq l\leq k))$ where:
\begin{enumerate}
 \item $A=X$,
 \item $\rel{A}{}=\{(v_1,v_2,\ldots, v_k):\{v_1,v_2,\ldots v_k\}\in R\hbox{ and }|\{v_1,v_2,\ldots, v_k\}|=k\}$,
 \item $\leq_\str{A}=\leq$.
 \item $F^\ell_\str{A}(\vec{t})$ is defined for every $t$-tuple $\vec{t}=(t_1,t_2,\ldots,t_t)$ without repeated vertices whenever $\lvert \bigcup \{M:T\subseteq M\in R\}\rvert=\ell$ and in this case $F^\ell_\str{A}(\vec{t})=\bigcup \{M:T\subseteq M\in R\}\rvert=N_\str{A}(\vec{t})$ where $T=\{t_1,t_2,\ldots ,t_t\}$.
\end{enumerate}

 Clearly $\str{A}\in \overrightarrow{P\mathcal D}_{kt\lambda}$ as it satisfies the above 4 conditions defining the class $P\mathcal D_{kt\lambda}$ and $\leq_\str{A}$ is a linear order. Note also that this correspondence is 1--to--1 as every $\str{A}\in \overrightarrow{P\mathcal D}_{kt\lambda}$ leads to an ordered partial $(k,t,\lambda)$-design $(A,\widetilde{R}_\str{A},\leq_\str{A})$.

The embeddings in $P\mathcal D_{kt\lambda}$ have the following meaning in the class of designs: $f:\str{A}\to \str{B}$ is an embedding if it satisfies (i), (ii) and the following:
\begin{enumerate}
\item[(iii')] Every $M\in \widetilde{R}_\str{B}\setminus f(\widetilde{R}_\str{A})$ intersects the set $f(\str{A})$ in at most $t-1$ elements (of course we have $f(A)=\{f(a):a\in A\}$ and $f(\widetilde{R}_\str{A})=\{f(M):M\in \widetilde{R}_\str{A}\}$).
\end{enumerate}
The equivalence (iii) and (iii') follows from our assumptions in definition of $P\mathcal D_{kt\lambda}$. In this case we the set $f(A)$ is closed in $\str{B}$. (Note that the condition (iii') is vacuous if $\str{A}$, $\str{B}$ are $(k,t,\lambda)$-designs.)

The following is the main result of this note:
\begin{thm}
For any $k\geq t\geq 2$ and $\lambda\geq 1$ the class $\overrightarrow{P\mathcal D}_{kt\lambda}$ is Ramsey.
\end{thm}
\begin{proof}(sketch)
We apply Theorem~\ref{thm:main} to the class $\overrightarrow{P\mathcal D}_{kt\lambda}$. Thus the only thing we have to check is the free amalgamation of $P\mathcal D_{kt\lambda}$. Thus let $\str{A},\str{B}_1,\str{B}_2$ be structures in $P\mathcal D_{kt\lambda}$, $\alpha_i:\str{A}\to \str{B}_i$ inclusion embeddings. Let $\str{C}=(C,\rel{C}{},(\func{C}{\ell}: k\leq \ell\leq K))$ be defined as follows:
$(C,\rel{C}{})$ is the free amalgam of relational structures $(B_1,\nbrel{\str{B}_1}{})$ and $(B_2,\nbrel{\str{B}_2}{})$ over $(A,\rel{A}{})$. For $\ell=k,k+1,\ldots, K$ we put $\func{C}{\ell}(T)=\nbfunc{\str{B}_i}{\ell}(T)$ whenever $\nbfunc{\str{B}_i}{\ell}(T)$ is defined (here we, without loss of generality, assume that embeddings $\beta_1$ and $\beta_2$ from the definition of amalgam are inclusions). 
 Note that this definition is consistent as $\alpha_i(A)$ is a closed set in $\str{B}_i$, $i=1,2$.
Thus $\str{C}$ is a free amalgam of $\str{A},\str{B}_1,\str{B}_2$ and Theorem~\ref{thm:main} applies.
\end{proof}
This theorem together with Theorem~\ref{thm:completion} implies Theorem \ref{thm:designs}.

\section{Remarks}
\noindent
\textbf{1.} The class $P\mathcal D_{kt\lambda}$ is the age of an ultrahomogeneous \Fraisse{} limit $\str{U}_{kt\lambda}$ which, by the Kechris, Pestov, Todor{\v c}evi{\' c} correspondence~\cite{Kechris2005}, is countable ``geometry-like'' structure with extremely amenable group of automorphisms and uniquely defined universal minimal flow (compare~\cite{Cherlin1999}).

\noindent
\textbf{2.} Note that the essential feature of the above proof is generality of Theorem~\ref{thm:main} and use of function (symbols) leading to the right definition of closed sets. It is easy to see that not closed subset do not form a Ramsey class (as, for example, one can distinguish subsets by their closures). 
Note also that by \cite{Evans3} the $(k,t,\lambda)$-designs have the ordering property, see also~\cite{Nevsetvril1995}.

\noindent
\textbf{3.} This proof and the relationship of designs and models has some further consequences and leads to interesting problems (such as the extension property for partial automorphisms (EPPA)), compare~\cite{Evans2}.

\bibliographystyle{abbrv}
\bibliography{ramsey.bib}
\end{document}